\documentclass[12pt]{amsart}

\usepackage{amssymb}
\usepackage{hyperref}

\usepackage{enumitem}

\usepackage{graphicx}

 \usepackage{tikz,tikz-cd}
\usetikzlibrary{calc,intersections}
\usepackage{multicol}

\makeatletter
\@namedef{subjclassname@2020}{%
  \textup{2020} Mathematics Subject Classification}
\makeatother

\usepackage[T1]{fontenc}

\newtheorem{theorem}{Theorem}[section]
\newtheorem{corollary}[theorem]{Corollary}
\newtheorem{lemma}[theorem]{Lemma}
\newtheorem{proposition}[theorem]{Proposition}

\newtheorem*{xque}{Question}
\newtheorem*{xTC}{Teissier's Criterion}

\newtheorem{mainthm}[theorem]{Main Theorem}

\theoremstyle{definition}

\newtheorem{remark}[theorem]{Remark}
\newtheorem{example}[theorem]{Example}

\numberwithin{equation}{section}

\newcommand{\Q}{\mathbb{Q}}
\newcommand{\C}{\mathbb{C}}
\newcommand{\F}{\mathbb{F}}

\newcommand{\A}{\mathcal{A}}
\newcommand{\Car}{\mathcal{C}}

\newcommand{\proj}{\mathbb{P}}

\DeclareMathOperator{\sing}{sing \, }
\DeclareMathOperator{\reg}{reg\,}
\DeclareMathOperator{\grad}{grad{}}
\DeclareMathOperator{\mult}{mult \,}
\DeclareMathOperator{\frakX}{\mathfrak {X}}

\begin{document}


\baselineskip=17pt


\title[{On Teissier's example}]{On Teissier's example of an equisingularity class that cannot be defined over the rationals}

\author[A. Parusi\'nski]{Adam Parusi\'nski}
\address {Universit\'e C\^ote d'Azur,  CNRS,  LJAD, UMR 7351, 06108 Nice, France}
\email{adam.parusinski@univ-cotedazur.fr}

\author[L. P\u aunescu]{Lauren\c tiu P\u aunescu}
\address{School of Mathematics and Statistics, The University of Sydney,
  Sydney, NSW, 2006, Australia }%
\email{laurentiu.paunescu@sydney.edu.au}%

\date{}
\begin{abstract}
A result of Teissier says that the cone over one of classical polygon examples in the real projective space gives, by complexification, a surface singularity 
which is not Whitney equisingular to a singularity defined over the field $\Q$ of rational numbers.  In this note we correct the example and give a complete proof {of Tesissier's result.}

\end{abstract}


\maketitle

\section{Introduction}
In the Comptes Rendus note \cite{Te90}, Bernard Teissier proposed an example of a surface singularity in $\C^3$ that is not equisingular, in the sense of Whitney, to a surface singularity in $\C^3$ defined by an analytic function with rational coefficients.  It is known by \cite{Ro18} that 
every complex singularity is equisingular, in a way even stronger than Whitney's, to a singularity defined by polynomial equations with coefficients in the field $\overline \Q$ of algebraic numbers.  Whether it is possible to deform it to a singularity defined over $\Q$ by a topologically equisingular deformation is still an open problem.

Teissier's example, taken from the book of Grünbaum \cite{Grunbook}, is one of the classical examples in convex geometry.  As B. Teissier let us know, this particular example admits a polynomial equation with rational coefficients and it has to be slightly corrected. This is explained in Section \ref{sec:grunbaum}. 

There is another slight problem in the argument of \cite{Te90}.  The above mentioned  property of the example is proven with the use of a theorem of \cite{LT79} on the deformation of a surface to its tangent cone.  That theorem is incorrect as stated.  This is explained in Section \ref{sec:deformation}.  An alternative proof of the result of \cite{Te90} is given in Section \ref{sec:proof}. 

\medskip
\noindent
\textbf{Notation.} If $\varphi$ and $\psi$ are non-negative real valued functions and we write $\varphi\le C \psi$ we mean that there is a real constant $C$ for which this inequality holds.  
This constant may occasionally change, but, for simplicity, it will still be denoted by $C.$
Sometimes we even abbreviate it to $\varphi\lesssim \psi$ meaning $\varphi\le C \psi$ for a constant $C$.

\section{Grünbaum's example}\label{sec:grunbaum}
On page 34 of Grünbaum's paper \cite{Grunpaper} (see also \cite[pp. 93--94]{Grunbook}), one finds the following two arrangements of 9 lines in the real projective space, built in an obvious way on the regular pentagon. 

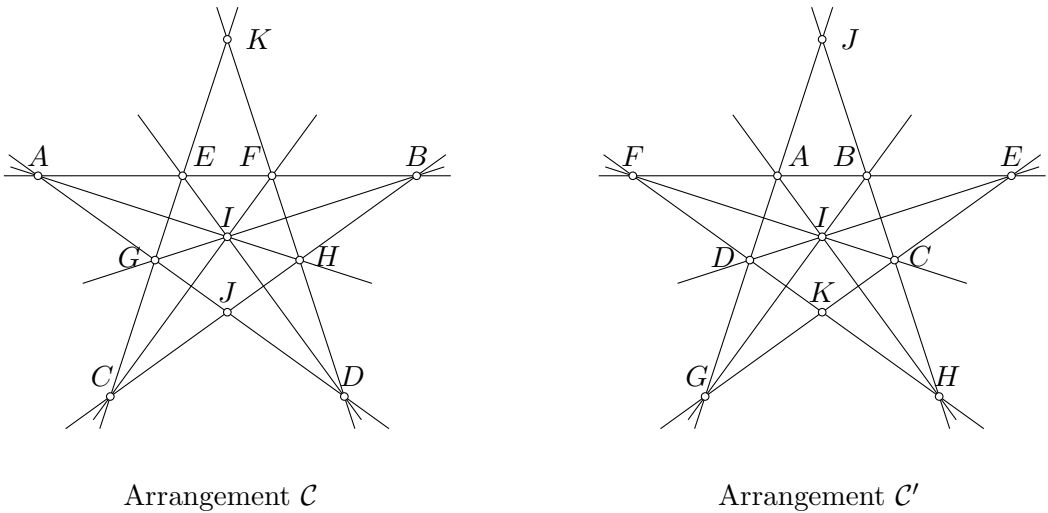
\begin{multicols}{2}
\bigskip
\tikzset{%
    add/.style args={#1 and #2}{
        to path={%
 ($(\tikztostart)!-#1!(\tikztotarget)$)--($(\tikztotarget)!-#2!(\tikztostart)$)%
  \tikztonodes},add/.default={.2 and .2}}}


 \begin{tikzpicture}
\coordinate (I) at (0,0);
\coordinate (H) at ({cos(342)},{sin(342)});
\coordinate (F) at ({cos(54)},{sin(54)});
\coordinate (E) at ({cos(126)},{sin(126)});
\coordinate (G) at ({cos(198)},{sin(198)});
\coordinate (J) at (0,-1);

\draw[name path=EF, add=2 and 2] (E) to (F);
\draw[name path=FH, add=2 and 2] (F) to (H);
\draw[name path=EG, add=2 and 2] (E) to (G);
\draw[name path=EI, add=1 and 3] (E) to (I);
\draw[name path=FI, add=1 and 3] (F) to (I);
\draw[name path=HI, add=1 and 3] (H) to (I);
\draw[name path=GI, add=1 and 3] (G) to (I);

\path [name intersections={of=EF and GI,by=B}];
\path [name intersections={of=EF and HI,by=A}];
\path [name intersections={of=FI and EG,by=C}];
\path [name intersections={of=EI and FH,by=D}];
\path [name intersections={of=EG and FH,by=K}];

\draw[name path=BH, add=.25 and 2] (B) to (H);
\draw[name path=AG, add=2 and .25] (G) to (A);

\draw[fill=white] (I) circle (0.05cm) node[above] {$I$};
\draw[fill=white] (H) circle (0.05cm) node[right,xshift=1.5,yshift=1.5] {$H$};
\draw[fill=white] (F) circle (0.05cm) node[above left] {$F$};
\draw[fill=white] (E) circle (0.05cm) node[above right] {$E$};
\draw[fill=white] (G) circle (0.05cm) node[left,xshift=-1.5,yshift=1] {$G$};
\draw[fill=white] (B) circle (0.05cm) node[above] {$B$};
\draw[fill=white] (A) circle (0.05cm) node[above] {$A$};
\draw[fill=white] (C) circle (0.05cm) node[above,xshift=-3] {$C$};
\draw[fill=white] (D) circle (0.05cm) node[above,xshift=3] {$D$};
\draw[fill=white] (J) circle (0.05cm) node[above] {$J$};
\draw[fill=white] (K) circle (0.05cm) node[right,xshift=3] {$K$};
\end{tikzpicture}\\ \\ 
 \hglue 2truecm Arrangement $\mathcal C$

\begin{tikzpicture}
\coordinate (I) at (0,0);
\coordinate (C) at ({cos(342)},{sin(342)});
\coordinate (B) at ({cos(54)},{sin(54)});
\coordinate (A) at ({cos(126)},{sin(126)});
\coordinate (D) at ({cos(198)},{sin(198)});
\coordinate (K) at (0,-1);

\draw[name path=AB, add=2 and 2] (A) to (B);
\draw[name path=BC, add=2 and 2] (B) to (C);
\draw[name path=AD, add=2 and 2] (A) to (D);
\draw[name path=AI, add=1 and 3] (A) to (I);
\draw[name path=BI, add=1 and 3] (B) to (I);
\draw[name path=CI, add=1 and 3] (C) to (I);
\draw[name path=DI, add=1 and 3] (D) to (I);

\path [name intersections={of=AB and DI,by=E}];
\path [name intersections={of=AB and CI,by=F}];
\path [name intersections={of=BI and AD,by=G}];
\path [name intersections={of=AI and BC,by=H}];
\path [name intersections={of=AD and BC,by=J}];

\draw[name path=EC, add=.25 and 2] (E) to (C);
\draw[name path=FD, add=2 and .25] (D) to (F);

\draw[fill=white] (I) circle (0.05cm) node[above] {$I$};
\draw[fill=white] (C) circle (0.05cm) node[right,xshift=1.5,yshift=1.5] {$C$};
\draw[fill=white] (B) circle (0.05cm) node[above left] {$B$};
\draw[fill=white] (A) circle (0.05cm) node[above right] {$A$};
\draw[fill=white] (D) circle (0.05cm) node[left,xshift=-1.5,yshift=1] {$D$};
\draw[fill=white] (E) circle (0.05cm) node[above] {$E$};
\draw[fill=white] (F) circle (0.05cm) node[above] {$F$};
\draw[fill=white] (G) circle (0.05cm) node[above,xshift=-3] {$G$};
\draw[fill=white] (H) circle (0.05cm) node[above,xshift=3] {$H$};
\draw[fill=white] (K) circle (0.05cm) node[above] {$K$};
\draw[fill=white] (J) circle (0.05cm) node[right,xshift=3] {$J$};
\end{tikzpicture}\\
\\
 \hglue 2truecm Arrangement {$\mathcal C' $}
\end{multicols}

As Grünbaum shows in \cite[Theorem 2.28]{Grunpaper}, any arrangement equiconfigurational with $\mathcal C$ (i.e. corresponding by a one-to-one bijection of lines) is projectively equivalent either to $\mathcal C$ or $\mathcal C' $,  
by a projective isomorphism that induces exactly the same one-to-one correspondence of lines and points. 
Moreover, this arrangement, or any arrangement equiconfigurational with it, cannot be defined over  $\Q$, that is, cannot  be given by  points with rational coordinates and  lines defined by equations with rational coefficients.  This follows by a direct computation as in the proof of \cite[Theorem 2.28]{Grunpaper}, or from its statement, because the cross-ratio of the four lines at the point $I$ (or the cross-ratio of the four points on the line $AB$) is irrational.  

\begin{example}\label{exmp:rational}
Consider the points of the arrangement $\mathcal C'$ as points of $\C$ as follows: 
$I=0$, $K=1$ and $A,B,C,D$ correspond to the primitive fifth roots of unity (i.e. we  consider the picture with the $x$-axis directed downwards). Let $G$ be the Galois group of the field $\F$, the extension of $\Q$ by the real coordinates of the fifth primitive roots of unity.  The extension has $[\F:\Q]=4$ and we may take $\sqrt{10+2\sqrt{5}}$ as its primitive element.  The set of fifth roots of unity and their  opposite numbers, i.e. the set of tenth roots of unity, is stable by the action of $G$.  

Let us add to the family of lines of $\mathcal C'$ the 10th line $\overline{JK}$.  
Denote by $l_i(x,y)= x - b_iy -a_i$, $i=1,2,. . . , 9$, the equations of the 9 lines in the arrangement $\mathcal C'$ and take 
$l(x,y)= y$ as the equation of the 10th one. The group $G$ acts on these lines by conjugating their coefficients.  Since the set of points defining these lines, and hence the set of the lines, is $G$ invariant, we conclude that the equation 
$l\prod_{i=1}^{9} l_i(x,y)$ is $G$ invariant, that is,  it is a polynomial with rational coefficients. As $l$ is fixed by $G$, we deduce that $ \prod_{i=1}^{9} l_i(x,y)$, 
a polynomial defining the union of the lines in $\mathcal C'$, has rational coefficients. 
\end{example}

To break the symmetry of the above arrangement 
we add the line $\overline{HE}$ to the arrangement $\mathcal C'$.   
The resulting arrangement consists of 10 lines intersecting at 18 points, including the point at infinity, the intersection of the lines $\overline {AD}$ and $\overline {HE}$.  Denote this arrangement by $\overline {\mathcal C'}$.   

Let $\A'$ be any affine line arrangement equiconfigurational with $\overline {\Car'}$.  
For simplicity of notation we label the points and the lines of $\A'$ by the same letters as the corresponding ones of $\overline {\mathcal C'}$.  
Let $l_i(x,y)$, $i=1, 2,. . . , 9$ be equations of the 9 lines of $\A'$ and  $l_{10}(x,y)$ an equation of the line $\overline{HE}$.  
Consider, similarly to Example \ref{exmp:rational},  the extension $\F$ of $\Q$ by the real coordinates of the points of the arrangement $\A'$ and denote by $G$ its Galois group ($\F$ and $G$ depend of course on $\A'$ and not only on $\overline {\Car'}$). Thus $G$ acts on the points of $\F^2$ by conjugating the coordinates and on the lines in $\F^2$ by conjugating the coefficients of their equations.


\begin{multicols}{2}
\tikzset{%
    add/.style args={#1 and #2}{
        to path={%
 ($(\tikztostart)!-#1!(\tikztotarget)$)--($(\tikztotarget)!-#2!(\tikztostart)$)%
  \tikztonodes},add/.default={.2 and .2}}}

\begin{tikzpicture}
\coordinate (I) at (0,0);
\coordinate (C) at ({cos(342)},{sin(342)});
\coordinate (B) at ({cos(54)},{sin(54)});
\coordinate (A) at ({cos(126)},{sin(126)});
\coordinate (D) at ({cos(198)},{sin(198)});

\draw[name path=AB, add=2 and 2] (A) to (B);
\draw[name path=BC, add=2 and 2] (B) to (C);
\draw[name path=AD, add=2 and 2] (A) to (D);
\draw[name path=AI, add=1 and 3] (A) to (I);
\draw[name path=BI, add=6.3 and 3] (B) to (I);
\draw[name path=CI, add=2 and 3] (C) to (I);
\draw[name path=DI, add=1 and 3] (D) to (I);
\draw[name path=EH, add=1.75 and 0.1] (E) to (H);

\path [name intersections={of=AB and DI,by=E}];
\path [name intersections={of=AB and CI,by=F}];
\path [name intersections={of=BI and AD,by=G}];
\path [name intersections={of=AI and BC,by=H}];
\path [name intersections={of=AD and BC,by=J}];
\path [name intersections={of=BI and EH,by=R1}];
\path [name intersections={of=CI and EH,by=R2}];
\path [name intersections={of=BC and DI,by=d}];
\path [name intersections={of=AD and CI,by=c}];
\draw[name path=EC, add=.25 and 2] (E) to (C);
\draw[name path=FD, add=2 and .25] (D) to (F);

\path [name intersections={of=AI and EC,by=a}];
\path [name intersections={of=BI and FD,by=b}];
\path [name intersections={of=FD and EC,by=K}];

\draw[fill=black] (I) circle (0.09cm) node[above] {$I$};
\draw[fill=red] (C) circle (0.07cm) node[right,xshift=1.5,yshift=1.5] {$C$};
\draw[fill=red] (B) circle (0.07cm) node[above left] {$B$};
\draw[fill=red] (A) circle (0.07cm) node[above right] {$A$};
\draw[fill=red] (D) circle (0.07cm) node[left,xshift=-1.5,yshift=1] {$D$};
\draw[fill=black] (E) circle (0.09cm) node[above left] {$E$};
\draw[fill=red] (F) circle (0.07cm) node[above] {$F$};
\draw[fill=red] (G) circle (0.07cm) node[above,xshift=-3] {$G$};
\draw[fill=black] (H) circle (0.09cm) node[above right,xshift=3] {$H$};
\draw[fill=white] (J) circle (0.05cm) node[right,xshift=3] {$J$};
\draw[fill=white] (R1) circle (0.05cm) node[right,xshift=3] {$R_1$};
\draw[fill=white] (R2) circle (0.05cm) node[above right,xshift=3] {$R_2$};
\draw[fill=white] (d) circle (0.05cm) node[above, xshift=4.5] {$d$};
\draw[fill=white] (c) circle (0.05cm) node[right,xshift=-0.5,yshift=2.5] {$c$};
\draw[fill=white] (b) circle (0.05cm) node[right] {$b$};
\draw[fill=white] (a) circle (0.05cm) node[right] {$a$};
\draw[fill=white] (K) circle (0.05cm) node[below] {$K$};
\end{tikzpicture}

\hglue 2truecm 
\vglue0.2truecm
\hglue 0.9truecm 
$R_\infty =$ intersection \\  \hglue 1.3truecm "at infinity" of the parallel \\
\hglue 1.3truecm  
lines $\overline{AD}$ and $\overline{HE}$\\ \\ \\ \\ 
\vglue2truecm
\hglue 0.9truecm  Arrangement {$\mathcal A' $} \\
\\ 
\end{multicols}

\begin{proposition}
The product $\varphi:=\prod_{i=1}^{10} l_i(x,y)$ is not $G$-invariant.  
\end{proposition}

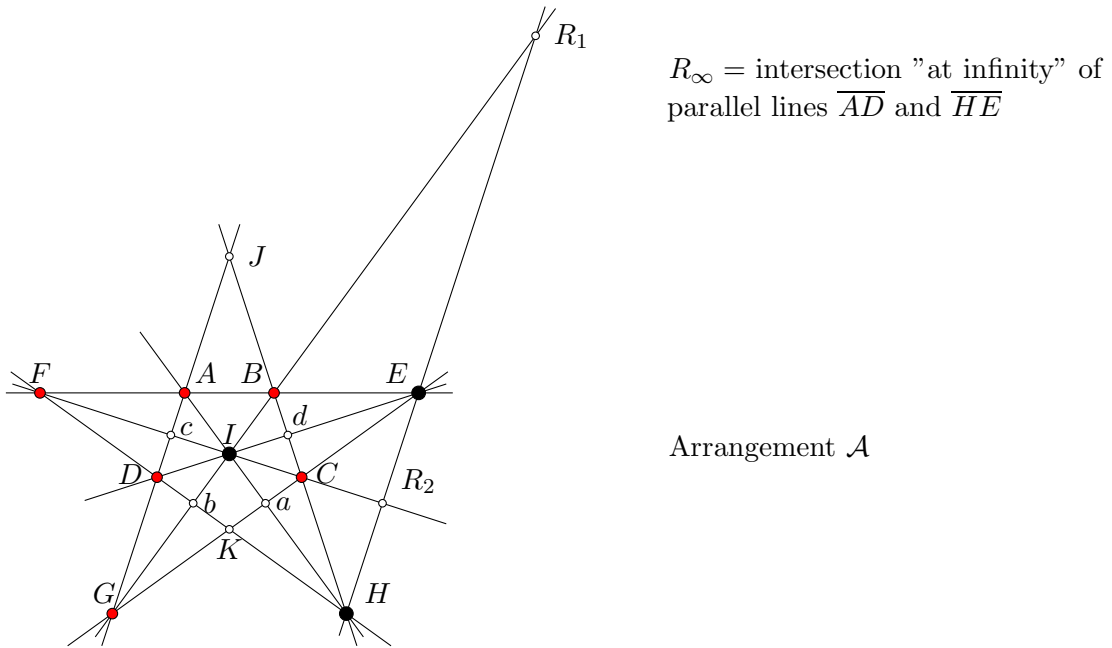
\begin{proof} 
Suppose for contradiction that the union of 10 lines is $G$-invariant.  Then so is the union of 18 intersection points.  The point at infinity cannot be moved to an affine point by an element of $G$.  For the other points and for all the lines we show that they are $G$-invariant as follows.
\begin{enumerate}
\item
To an affine point $P$ of $\mathcal A'$ we associate its weight, i.e. the number of lines intersecting at $P$.  There are 3 possibilities: 
\begin{enumerate}
\item
4 lines intersect at $I$, $H$, $E$;  these points are marked in black.
\item
3 lines intersect at $A$, $B$, $C$, $D$, $F$, $G$;  these points are marked in red.
\item
2 lines intersect at $a$, $b$, $c$, $d$, $K$, $J$, $R_1$, $R_2$;  these points are marked in white.
\end{enumerate}
The group $G$ preserves each of these three sets of points. 
\item
The following lines are stable by the action of $G$:  $\overline{FE}$ as the only line without points of  
weight 2,  $\overline{HE}$ as the only line without points of weight 3, $\overline{AD}$ as the only line without points of weight 4.
\item
The intersection points of these lines, $A$, $E$ (and $R_\infty$), are stable by $G$.
\item
Each of the points of weight 4, $I$, $H$, $E$, is stable by $G $ (the line $\overline{HE}$ and $E$ are both stable).  
\item
Lines containing two stable points are stable.  Thus the lines $\overline{IE}$ and 
$\overline{IH}$ are stable.  Hence the intersection points of these lines with the other stable lines are stable by $G$.  Therefore $D$, $A$, and $H$ are stable. 
\item
All points of the line $\overline{AD}$ are stable ($J$ is of weight 2 and $G$ is of weight $3$).
\item
All points of the line $\overline{HE}$ are stable ($R_1$ is on a stable line $\overline{IG}$).
\item
Every line of the arrangement, different from $\overline{AD}$ and  $\overline{HE}$,  joins a point of $\overline{AD}$ and a point of $\overline{HE}$.  Therefore, by (6) and (7), all lines are stable by $G$. 
\end{enumerate}
Hence all points of $\mathcal A'$ have rational coefficients. Consequently,  all points of  $\mathcal C'$ have rational coefficients.  This contradicts the conclusion of \cite[Theorem 2.28 and the discussion on page 34]{Grunpaper}.  
\end{proof}

\begin{remark}
Let $\mathcal B$ be a projective line arrangement equiconfigurational with $\A'$.  Then $\mathcal B$ is projectively isomorphic to $\A'$ or to  the arrangement $\A$ obtained from $\Car$ by adding the line $\overline{HE}$.  Indeed, suppose that $\mathcal B$ and $\A'$ are equiconfigurational by a bijection of lines denoted by $\Phi$. Let 
$\tilde {\mathcal B}$ denote the arrangement obtained from $\mathcal B$ by removing the line $\Phi^{-1}(\overline{HE})$.  Then $\tilde {\mathcal B}$ is equiconfigurational to $\Car'$, hence, by \cite[Theorem 2.28]{Grunpaper}, projectively isomorphic to 
either $\Car'$ or $\Car$.  On these 9 lines this isomorphism induces exactly the same bijection as $\Phi$ and therefore it sends the line $\overline{HE}$ to the corresponding line.  
\end{remark}

\section{Proof of Teissier's result}\label{sec:proof}

Let $l'_i (x,y) -a_i=0$, $i=1,2,. . . , 9$, $l'_i$ being a linear form, be the equations of the 9 lines defining the arrangement $\mathcal C'$ of the previous section.  Then $f_0(x,y,z):= \prod_i \tilde l_i(x,y,z)$, where $\tilde l_i(x,y,z)=: l_i'(x,y) -a_i z$, defines the cone over the union of these lines. Consider this cone as a complex singularity in $\C^3$. In \cite{Te90}  Teissier claims the following result.   

\begin{proposition} [\cite{Te90}]\label{pr:teissier}
Let $f(x,y)=0, \text{with}\,\, x\in \C^3, y\in \C^p$, be an analytic deformation of $f_0$ that is Whitney equisingular along $\{0\}\times \C^p$.  Then the family of tangent cones to the fibres $y\to X_y:=\{(x,y); f(x,y)=0\}$ is also Whitney equisingular.  

As a corollary,  one finds that in Teissier's example, each of these tangent cones $C_{X_y,y}$ is the union of $9$ planes that is equivalent to $X_0$ by a projective transformation, hence cannot be defined over $\Q$.  
\end{proposition}

Besides the fact that $\mathcal C'$ has to be replaced by the arrangement $\overline {\mathcal C'}$, 
Teissier's proof of the first claim uses 
Théorème (2.1.1) 
of \cite{LT79} which is incorrect as stated.  We give below an alternative proof of this first claim, for the arrangement $\overline {\mathcal C'}$ or any arrangement of lines, that avoids using Théorème (2.1.1) of \cite{LT79}; see Theorem \ref{thm:main}.  In Section \ref{sec:deformation} we provide a counterexample to 
Théorème (2.1.1) of \cite{LT79}.

Consider a family of surface singularities of constant multiplicity $d$, 
\begin{align*}
f(x,y)=f_d(x,y)+ f_{d+1}(x,y) +  \cdots  , \quad f(0,y)\equiv 0,
\end{align*}
 where $x=(x_1,x_2,x_3)\in(
\mathbb C^3,0), y\in (\C^p,0),$ and the $f_i$ are homogeneous in $x$ of degree $i$. 
 Let us write $X_y:=\{(x,y); f(x,y)=0\}$ for $y\in \C^p$. 

We denote $X=\{f=0\}, \, \Sigma=X_{\sing}$, $Y=\{x=0\}$, and the regular part of $X$ by $X_{\reg}$.
Let $C_{X,Y}$ denote the normal cone of $X$ along $Y$, that is  $C_{X,Y} = f_d^{-1} (0) $.  
Similarly,  $C_{\Sigma,Y}$  will represent the normal cone of $\Sigma$ along $Y$. 

\subsection{Exceptional tangents}
First we recall the notion of exceptional tangents of a surface singularity $S\subset (\C^3,0)$.  It was introduced in \cite{HL75} 
where the following result was shown; see also the survey paper \cite{LeSn21}.

\begin{proposition}[{\cite[Théorème 3.1]{HL75}}]\label{prop:HL}
Let $S\subset (\C^3,0)$ be an isolated surface singularity.  Then the set of the limits of planes tangent to $S$ at $0$ is the union of the set of hyperplanes tangent to the tangent cone $C_{S,0}$ and the set of hyperplanes of a finite number of line-pencils whose axes are lines of $C_{S,0}$, called the exceptional tangents. 

Moreover, among these lines we find all the lines of the singular set of the reduced tangent cone of $S$ at $0$.
\end{proposition}

This proposition follows from a theorem of Teissier \cite[Remarque 1.6]{Te73} that describes the limit of tangent hyperplanes to an isolated hypersurface singularity $(S,0)\subset (\C^n,0)$, of arbitrary dimension, in terms of the following criterion:

\begin{xTC}
A hyperplane $H$ is not a limit of tangent hyperplanes if and only if  $H\cap S$ is an isolated singularity with  the minimum Milnor number among all the Milnor numbers of the intersections of $S$ with hyperplanes.
\end{xTC}

An analog of the main statement of Proposition \ref{prop:HL} in the non-isolated case was showed by  Lê  \cite{Le81}.  We will also need  a version of the last part of Proposition \ref{prop:HL} which is valid for non-isolated singularities.  This we show in Proposition \ref{cor:nonisolHL} below. 

 \begin{lemma}\label{lem:nonisolates}
Let $S\subset (\C^n,0)$ be a hypersurface, $S=f^{-1} (0)$ with $f$ reduced, and let $\ell$ be a line through the origin in $\C^n$ that is not tangent to $S_{\sing}$.  Then, for a generic linear form $\varphi: \C^ n\to \C$ and for $N$ sufficiently large,   
$g(x):= f(x) + \varphi^N(x)$ satisfies 
\begin{enumerate}[label=\roman*)] 
\item  
$(g^{-1} (0))_{\sing}= S_{\sing} \cap \varphi^{-1} (0)$; 
\item 
for a sufficiently small conical open neighbourhood $U\subset \C^n$ of $\ell^*:= \ell \setminus \{0\}$,  
the limits as $x\to 0$, $x\in U$,  of hyperplanes tangent at $x$ to the levels of $f$ and those to the levels of $g$  coincide.  
\end{enumerate}
 \end{lemma}

 We say  that a subset of $\C^n$ is \emph{conical} at the origin if it is the intersection of a $\C^*$-homogeneous set, with respect to  the standard action of $\C^*$ on $\C^n$: $\C^*\times \C^n \ni (s,x)\to sx \in \C^n $,  and an open neighbourhood of the origin. 

  \begin{proof}[Proof of Lemma \ref{lem:nonisolates}] 
Let the system of coordinates $x=(x_1, \ldots , x_n)$ be such that  
\begin{align}\label{eqn:polar}
\left \{ \frac{\partial f}{\partial x_1}= \frac{\partial f}{\partial x_2}
= \cdots =\frac{\partial f}{\partial x_{n-1} }=0 \right \} = \Gamma \cup 
S_{\sing}, 
\end{align} 
with $\Gamma$ being of dimension $1$ (or empty).  Then $\Gamma$, a generic relative polar curve,  depends only on the projection $(x_1, \ldots , x_{n-1}) : \C^n \to \C^{n-1}$.  Therefore, after changing $x_n$ if necessary, we may assume moreover that $\{x_n=0\} \cap (\Gamma \cup \ell) = \{0\}$ and   
denote $H:=\{x_n=0\}$.  
Let $U$ be a conical neighbourhood of $\ell^*$ such that  $\overline U \cap (H\cup C_{S_{\sing},0}) = \{0\}$, $0\notin U$.  Then, by the {\L}ojasiewicz inequality, there exist a positive integer $N$ and a constant $c>0$ such that 
\begin{align}\label{eqn:loj}
\|\grad f (x)\| \ge c |x_n|^{N-2} \quad \text { for } x\in U.
\end{align}

Let $g(x):= f(x) + x_n^N$.  Then 
\begin{align}\label{eqn:gandf}
\frac{\partial g}{\partial x_i} = \frac{\partial f}{\partial x_i} \text { for } i= 1, \ldots  ,n-1 
\text { and } \frac{\partial g}{\partial x_n} = \frac{\partial f}{\partial x_n} + Nx_n^{N-1}. 
\end{align}
Therefore, because $H\cap \Gamma = \{0\}$, we have $(g^{-1} (0))_{\sing}= S_{\sing} \cap \varphi^{-1} (0)$ as needed.  By \eqref{eqn:loj} and \eqref{eqn:gandf}, $g$ on $U$ satisfies the {\L}ojasiewicz inequality similar to \eqref{eqn:loj} 
 and therefore 
$$
\left | \frac{\grad f}{\|\grad f\|} -\frac{\grad g}{\|\grad g\|} \right | \le C |x_n| .
$$
Then $\varphi (x) = x_n$ and $N$ given by \eqref{eqn:loj} guarantee  the conclusions of the lemma. 
 \end{proof}

 \begin{corollary}\label{cor:nonisolHL}
Let $S\subset (\C^3,0)$ be a surface singularity. Suppose that a line $\ell$ through the origin in $\C^3$ is not tangent to $S_{\sing}$ but it is included in the singular set of the reduced tangent cone of $S$ at $0$.  
Then the planes containing $\ell$ form a pencil of limits of tangent planes to $S_{\reg}$, that is, these planes are exceptional tangents.  
\end{corollary}

\begin{proof}
Let $S=f^{-1} (0)$, with $f$ reduced, and suppose $\dim S_{\sing} =1$.  Assume that $g(x):=f(x) + \varphi^N(x)$ satisfies the conclusion of the Lemma \ref{lem:nonisolates}, and moreover, by (i) of this lemma, we may choose $\varphi$  such that  $\tilde S= g^{-1}(0)$ has isolated singularity at the origin.  We may also choose $N$ sufficiently large so that the multiplicities at the origin satisfy $\mult_0 f= \mult_0 g <N$.  Then, the normal cones $C_{S,0}$ and $C_{\tilde S, 0 }$ are equal and given by the same equation.  

By Proposition \ref{prop:HL}, the planes containing $\ell$ form a pencil $\Pi$ of limits of tangent planes to $\tilde S_{reg}$.  To see that they form a pencil of limits of tangent planes to $S_{reg},$ it is sufficient to note that $\Pi$ can be obtained by limits of tangent planes to $\tilde S_{\reg}$ on sequences of points $x\to 0$ with the secant lines $\overline {x0} \to \ell$, that is, the points from the set $U$ of the conclusion (ii) of Lemma  \ref{lem:nonisolates}.  The latter claim comes from the description of the common limits of secants and tangent hyperplanes to any local singularity, as the union of dual correspondences, i.e. subsets of  $\mathbb P^{n-1}\times \check{\mathbb P}^{n-1}$, $\mathbb P^{2}\times \check{\mathbb P}^{2}$ in our case,  that are projectivised conormal spaces to algebraic subsets of $\mathbb P^{n-1}$ (or to $\check{\mathbb P^{n-1}}$ by duality).  This follows, for instance, from \cite[Théorème 2.1.1]{LT88}, in the case where the small stratum $Y$ is just the origin.  
\end{proof}

\subsection{Main Theorem}

\begin{mainthm}\label{thm:main}
Let $f(x,y)$ with $x\in \C^3, y\in {\C}^p,$ be a reduced analytic function germ of multiplicity $d$ at the origin.  Denote $X=f^{-1}(0)$, $\Sigma = X_{\sing}$, $Y= \{0\}\times \C^p$.  
Assume that $f(x,0)= f_d(x,0)$ and that $(X_{\reg}, \Sigma\setminus Y,Y)$ is a Whitney stratification of $X$.  
Then the following hold:
\begin{enumerate}[label=\roman*)] 
\item  
For all $y$, $X_y= \{x\in \C^3; f(x,y)=0\}$ has no exceptional tangents at $x=0$;
\item 
$((C_{X,Y})_{\reg}, Y)$ satisfies Whitney's conditions;
\item  
$(C_{X,Y})_{\sing}=C_{\Sigma,Y}$;
\item
$C_{\Sigma,Y} \setminus Y$ is smooth;
\item 
$(C_{\Sigma,Y} \setminus Y, Y)$ satisfies Whitney's conditions;
\item $((C_{X,Y})_{\reg}, C_{\Sigma,Y}\setminus Y, Y)$
is a Whitney stratification of the normal cone  $C_{X,Y}$.
\end{enumerate}
\end{mainthm}

\begin{proof} 
{\bf Step 1.} We use \cite[Théorème 2.1.1]{LT88} to show \emph{i)}, \emph{ii)}, and \emph{v)}, the latter under the assumtpion that \emph{iii)} and \emph{iv)} hold. (It is possible to reduce to the case $\dim Y=1$ using the argument of the proof of \cite[Proposition 1.2.1]{Te82}, but this is not necessary in our case.)  

By the assumption that $(X_{\reg},Y)$ satisfies Whitney's conditions, the family of  limits of secants and tangent hyperplanes to the fibers over $Y$
is equidimensional over $Y$. It contains the (projective) conormal to the tangent cone, that is, 
\begin{align*}
\overline{\{(x,H,y)\in \C^3 \times \check \proj^2 \times {\C}^p ;(x,y)\in (C_{X,Y})_{\reg}, H=T_x((C_{X_y,y})_{\reg})\}} , 
\end{align*} 
and the conormals to the subcones of $C_{X,Y}$, denoted by $Y(V_\alpha)$ in \cite{LT88} and  called the exceptional cones (see {Définition} 2.1.4 ibid).  By Whitney's conditions the conormal spaces to the exceptional cones are also equidimensional over $Y$, and therefore by the homogeneity assumption for $y=0$,  i.e. $X_0= C_{X_0,0}$, they are empty.   This shows {that for all $y,$ 
 $X_y$ has} no exceptional tangents, which is \emph{i)}.

Then, the reciprocal argument (see \cite[Théorème 2.1.1]{LT88}) for $V_{\alpha}=C_{X,Y}$ shows that the pair of strata $((C_{X,Y})_{\reg}, Y)$ satisfies Whitney's conditions, proving \emph{ii)}.
(This also follows by an elementary computation: see Lemma \ref{lem:observation}
below.) 

A similar argument shows \emph{v)}, provided that \emph{iii)} and \emph{iv)}  hold. 
Because $(\Sigma\setminus Y,Y)$ satisfies Whitney's conditions, by \cite[Theorem 2.1.1]{LT88} the set
$$
W=\overline{\{([x],H,y);x\in (C_{\Sigma,Y})_{\reg}, T_x((C_{\Sigma,Y})_{\reg})\subset H}\} 
\subset \mathbb P^2\times \check{\mathbb P}^2\times \C^p
$$
 is equidimensional over $Y$,
and $((C_{\Sigma,Y})_{\reg}, Y)$ satisfies Whitney's conditions.

\medskip

\noindent {\bf Step 2.} 
We show the third and the fourth claims of Theorem \ref{thm:main}. 

By definition, for an ideal $I$ the normal cone of $V(I)$ along $Y$ is  
$C_{V(I),Y}=V(I_{in})$, where $I_{in}:=(g_{in}, g\in I)$.  Here $g_{in}$ denotes the homogeneous initial form of $g$ with reespect to the variable $x$.  
Thus $C_{X,Y} = V(f_d)$ and 
$$
(C_{X,Y})_{\sing}=V((\frac{\partial f_d}{\partial x_1}, \frac{\partial f_d}{\partial x_2}, ..., \frac{\partial f_d}{\partial x_n}, \frac{\partial f_d}{\partial y_1}, \ldots  , \frac{\partial f_d}{\partial y_p}))
. $$
The singular locus of $X$ is defined as $\Sigma =V(J_{f})$, where 
$$
J_{f}=(\frac{\partial f}{\partial x_1}, \frac{\partial f}{\partial x_2}, ..., \frac{\partial f}{\partial x_n}, \frac{\partial f}{\partial y_1}, \ldots , \frac{\partial f}{\partial y_p})).
$$
 Its normal cone along $Y$ is $C_{\Sigma,Y}=V((J_f)_{in})$.   
Clearly $\frac{\partial f_d}{\partial x_i}$ is either zero or the initial part of 
$\frac{\partial f}{\partial x_i}$, and similarly for $\frac{\partial f_d}{\partial y_j}$.  
This shows the inclusion $C_{\Sigma,Y}\subset (C_{X,Y})_{\sing}$.

Note that $X_0$ is reduced because so is $X$, and both the regular and the singular part of $X $ are Whitney equisingular along $Y$.  Therefore, since $X_0$ is reduced and homogeneous,  its singular part $\Sigma_0$  is a finite union of lines in $C_{X_0,0}$.  Then, by \cite{Hi70}, because $(\Sigma \setminus Y, Y)$ satisfies Whitney's conditions, $\Sigma$ is normally flat (i.e. equimultiple) along $Y$, and $\Sigma$ is topologically trivial along $Y$.  This shows that $\Sigma_0= \Sigma\cap \{y=0\}$ and that $\Sigma$, with reduced structure, is a finite union of mutually transverse smooth subvarieties intersecting along $Y$.  Therefore $C_{\Sigma,Y}$  is a finite union of smooth families of lines intersecting only along $Y$. 
In particular, $C_{\Sigma,Y}\setminus Y$ is smooth,  and this shows \emph{iv)}. 

As shown before, $C_{\Sigma,Y} \subset (C_{X,Y})_{\sing}$, and hence,  for every $y$,  
$$ 
(C_{\Sigma,Y})_y \subset ((C_{X,Y})_{\sing})_y \subset (C_{X_{y},y})_{\sing} . 
$$
 By Corollary \ref{cor:nonisolHL},   
  $(C_{X_{y},y})_{\sing}\setminus C_{\Sigma_y,y}$ is contained in the set of exceptional tangents, empty in our case. This shows that $(C_{X,Y})_{\sing}=C_{\Sigma,Y},$ i.e.
the third claim of Theorem \ref{thm:main}.   This, as explained before, completes the proof of \emph{v)}.

\bigskip
\noindent
{\bf Step 3.}
To complete the proof of Theorem \ref{thm:main} we show that 
\begin{align}\label{eqn:lastWhitney}
((C_{X,Y})_{\reg}, C_{\Sigma,Y} \setminus Y)
\end{align} 
satisfies Whitney's conditions.  In the proof we will use the standard 
action of $\C^*$ on $\C^3$.

Let $\ell$  be a line in $(C_{X_0,0})_{\sing}$,  say given by $x_1=x_2=0$.
Denote by $L$ the component of $C_{\Sigma,Y}$  containing it and
let $N=\{x_3=c\}$  be a transverse slice of $\ell$.  Then,  to show that  \eqref{eqn:lastWhitney}
 is Whitney, it is equivalent to show that 
 \begin{align}\label{eqn:lastWhitneyslice}
(N\cap (C_{X,Y})_{\reg}, N\cap C_{\Sigma,Y} \setminus Y)
\end{align} 
is Whitney. Indeed it follows from the fact that $\ell^* = \ell \setminus \{0\}$ is an orbit of the $\C^*$ action.  

Note that $N\cap C_{X,Y}$  is a family of plane
curve singularities along $N\cap L$.  Therefore, to show that it
is Whitney it suffices to show that this family is $\mu$  constant (see \cite{LeRa76} or \cite{Te76})  
or equivalently that 
 \begin{align}\label{eqn:muconstant}
\left |\frac{\partial f_d}{\partial y} \right| \leq C \left| \frac{\partial f_d}{\partial x_1} , \frac{\partial f_d}{\partial x_2} \right|
\end{align} 
in a neighbourhood of $\ell\cap N$ in $N$.

Note also that \eqref{eqn:muconstant} has to hold in a whole neighbourhood of $\ell$ 
in $N,$ not merely on the zero set of $f_d$. Therefore to show it we consider the strict Thom stratifications of $f$ and of $f_d$.  A strict Thom stratification of $f$ is a Whitney stratification of $X$ satisfying the strict Thom condition $w_f$ along each stratum (see \cite{HMS84}).  In particular, along the smallest stratum $Y$ this condition reads
\begin{align}\label{eqn:thomforf}
\left |\frac{\partial f}{\partial y}\right |\leq C|x|\left|\frac{\partial f}{\partial x}\right|.
\end{align}
By \cite{Pa93} or \cite{BMM94}, every Whitney stratification of $X$ satisfies the $w_f$ condition along every stratum.  In particular, we may assume that  \eqref{eqn:thomforf} holds. 

Recall that we say that a subset of $\C^3$ is \emph{conical} if it is the intersection of a $\C^*$-homogeneous set and an open neighbourhood of the origin.  

\begin{lemma}\label{lem:localinequality}
There is an open conical neighbourhood $U$ of $\ell^*:= \ell\setminus \{0\}$ in $\C^3$  such that 
\begin{align}\label{eqn:prepa}
\left|\frac{\partial f}{\partial x_3}\right|\leq C \left|\frac{\partial f}{\partial x_1}, \frac{\partial f}{\partial x_2} \right| 
\end{align}
on $U\times \C $.  
\end{lemma} 

\begin{proof}
 First note that, for $y=0$, \eqref{eqn:prepa} holds on an open conical neighbourhood of $\ell^*$ in 
 $\C^3$. Indeed, it follows from the homogeneity of $f_d$ that there is a Whitney stratification of $X_0$ with $\ell^*$ as a stratum, and both sides of \eqref{eqn:prepa} are homogeneous of the same degree.  

To show that \eqref{eqn:prepa} holds also for $y\ne 0$ and small, maybe with a slightly different constant $C$, we may use the fact that if 
a stratification satisfies  the $w_f$ condition then the exceptional divisor $E\subset {\mathbb P}^2\times \check{\mathbb P}^2\times \mathbb C$ of the blowing-up of the ideal $(x_1,x_2,x_3)(\frac{\partial f}{\partial x_1}, \frac{\partial f}{\partial x_2},  \frac{\partial f}{\partial x_3})$  
 is equidimensional, see  \cite[Proposition 3.3.1, Théorème 6.1, and Remarque 6.2.1]{HMS84}. Recall that the line $\ell$ is given by $x_1=x_2=0$.  Thus the conclusion of Lemma \ref{lem:localinequality} means that 
  the point $([0:0:1],[0:0:1]) \notin E$ for $y=0$.  But, if this is the case for $y=0$, then it holds as well for $y$ close to $0$.  
 \end{proof}

 \begin{lemma}\label{lem:observation}
Let  
\begin{align*}
F(x,y,t)=t^{-d}f(tx,y) = f_d(x,y) + t f_{d+1}(x,y) + \cdots 
\end{align*}
 be the deformation to the normal cone $C_{X,Y}$ induced by $f$. 
If $f$ satisfies \eqref{eqn:thomforf}, then the induced deformation  to the normal cone satisfies  
\begin{align}\label{eqn:thomforF}
\left|\frac{\partial F}{\partial y}\right|\leq C|x| \left|\frac{\partial F}{\partial x}\right|.
\end{align}
Similarly, if for a conical set $U\subset \C^3$ the inequality  
\begin{align}\label{eqn:thomforf2}
\left|\frac{\partial f}{\partial y}\right|\leq C|x| \left|\frac{\partial f}{\partial x_1}, \frac{\partial f}{\partial x_2} \right| ,
\end{align}
holds on $U$,  then the analogous inequality holds on $U\times\C\times 
\overline D_1$ for $F$,   
\begin{align}\label{eqn:thomforF2}
\left|\frac{\partial F}{\partial y}\right|\leq C|x| \left|\frac{\partial F}{\partial x_1}, \frac{\partial F}{\partial x_2} \right|.
\end{align}
(Here $\overline D_1$ denotes $\{|t|\le 1\}$.)
\end{lemma}

\begin{proof}
Using the action of $s\in \C^*$ we have 
\begin{align*}
& \frac{\partial F}{\partial x}(sx,y,t)=s^{d-1}\frac{\partial F}{\partial x}(x,y,st), \quad
 \frac{\partial F}{\partial y}(sx,y,t)=s^{d}\frac{\partial F}{\partial y}(x,y,st). 
\end{align*}
Then, by \eqref{eqn:thomforf}, 
 \begin{align*} 
\left|
\frac{\partial F}{\partial y}(x,y,s)\right | & =|s^{-d}|\left |\frac{\partial F}{\partial y}(sx,y,1)\right | \\ & \lesssim |s^{-d}||sx|\left |\frac{\partial F}{\partial x}(sx,y,1)\right | = |x| \left |\frac{\partial F}{\partial x}(x,y,s)|\right|.
\end{align*}
The proof of \eqref{eqn:thomforF2} is similar.  
\end{proof}

Note that \eqref{eqn:thomforf} and \eqref{eqn:prepa} give \eqref{eqn:thomforf2}, and hence, by Lemma \ref{lem:observation}, \eqref{eqn:thomforF2}. The latter gives 
\begin{align*}
\left |\frac{\partial f_d}{\partial y} \right| \leq C |x| \left| \frac{\partial f_d}{\partial x_1} , \frac{\partial f_d}{\partial x_2} \right| ,
\end{align*}
which implies \eqref{eqn:muconstant}.  This shows \emph{vi)} of the Theorem \ref{thm:main}
and completes the proof.  
\end{proof}

\section{When the deformation to the tangent cone is equisingular?}\label{sec:deformation}

Let $X:=\{f(x,y,z)=0\}\subset (\C^3,0)$ be a surface singularity.  Suppose that the tangent cone $C_{X,0}$ of $X$ at the origin is reduced. Let $F(x,y,z,t)$ denote the deformation of $X$ to the tangent cone $C_{X,0}$ and denote $F^{-1}(0)$ by $\frakX $. Then Théorème (2.1.1) of \cite{LT79} claims that 
the following two conditions are equivalent.
\begin{enumerate}
  \item
$X$ has no exceptional tangents.
\item
The deformation of $X$ to the tangent cone $C_{X,0}$ is equisingular.
\end{enumerate}
There are various notions of equisingularity stated in this theorem and claimed equivalent in this case.  One of them, the condition b) of Théorème (2.1.1) of \cite{LT79}, says that the partition of $\frakX$: $(\frakX - \Sigma, \Sigma \setminus \{0\} \times \C, \{0\} \times \C)$, where $\Sigma:=\frakX_{\sing}$, is a Whitney stratification of $\frakX$. 

Suppose that $X$ has no exceptional tangents.
Then $(\frakX - \Sigma, \Sigma \setminus \Sigma_{\sing})$ satisfies Whitney's conditions, i. e. the condition a) of loc. cit.  This part of the proof is correct and was generalized to the hypersurfaces of arbitrary dimensions in \cite{Flo13}; see also \cite{FT18}. 

Both the counterexample below and our proof of the previous section suggests that Théorème (2.1.1) of \cite{LT79} may be true under some additional assumptions.

\begin{xque}
Suppose moreover that the (reduced) singular locus of $X$ is a finite union of smooth curves mutually not tangent.  Is then $(\frakX - \Sigma, \Sigma \setminus \{0\} \times \C, \{0\} \times \C)$ a Whitney stratification of $\frakX$? 
\end{xque}

Note that Théorème (2.1.1) of \cite{LT79} holds true if $X$ is an isolated singularity.  

\subsection{Example of a singularity that is not equisingular to its tangent cone}
Let 
$$f(x,y,z)=z(zx-y^2)+zx^3=z(zx-y^2+x^3), \quad (x,y,z)\in \mathbb C^3.$$
  Then $X=\{f=0\}$ is the union of the plane $X_1=\{z=0\}$ and the Morse singularity (ordinary double point) $X_2=\{zx-y^2+x^3=0\}$ that is analytically isomorphic to $zx-y^2=0$ (by the change $(x,y,z)\to (x,y,\tilde z=z+x^2), zx-y^2+x^3=(z+x^2)x-y^2=\tilde z x-y^2$).  
We note that 
\begin{enumerate}
  \item
neither $X_1$ nor $X_2$ has exceptional tangents and hence neither has $X$;
\item
the singular locus of $X$ is  $X_1\cap X_2=\{z=zx-y^2+x^3=0\}$, i.e. the cusp $y^2=x^3, z=0$.
\end{enumerate}

\bigskip
\noindent
{\bf Deformation to the normal cone.} 
Consider the deformation to the normal cone of $f$
\begin{align}\label{eq:deform}
F(x,y,z,t)=z(zx-y^2)+tzx^3=z(zx-y^2+tx^3).
\end{align}
Let
$\frakX:=\{F=0\}=\frakX_1\cup \frakX_2$, where $\frakX_1=\{z=0\}$ and $\frakX_2=\{zx-y^2+tx^3=0\}$. Then $\frakX$ satisfies the following properties:
\begin{enumerate}
  \item
Both $\frakX_1 $ and $ \frakX_2$ are analytically trivial along $ \{0\}\times \C$.
 \item
$\frakX_{\sing} =\frakX_1\cap \frakX_2=\{z=y^2-tx^3=0\}$ is the deformation of the cusp to a double line.
 \item
  $(\frakX_{reg} ,\{0\}\times \C)$ satisfies Whitney's conditions (because the components are analytically trivial along  $ \{0\}\times \C$).
 \end{enumerate}

 The deformation $t\to \frakX_t$ is not equisingular. The singular locus of $\frakX$, considered as a function of $t$, splits by (2).  It is not topologically trivial either. If $t\to \frakX_t$  were topologically trivial, then the induced deformation on the normal section by $N=\{x=1\}$,  at a singular point of $C_{X,0}$,
\begin{align}\label{eq:deformsliced}
h(y,z,t)=z( z-y^2+t) ,
\end{align} 
would be topologically trivial as a deformation of plane curves, which is not the case.  

Note that the discriminant of $(y,z,t)\to (t, y-bz) $ restricted to the zero set of $h$ equals 
 $(y^2-t)^2$ (up to a non-zero unit). Thus the sum of the multiplicities of its zero set is
independent of $t$ and equals $4$.  Recall after \cite[Chap. II, Prop 1.2]{Te76} the  formula for the multiplicity of discriminant:  $\mult \Delta = \sum_i (\mu_i+(m_i-1))$, where $\mu_i$ denote the Milnor numbers and $m_i$ the multiplicities 
of the singular points of $h$.  Then, for $t=0$, the fiber is $z(z-y^2)=0$, with a single singular point $z=y=0$, its multiplicity and its Milnor number are  $ m=2, \mu =3$, respectively, whilst for $t\ne 0$ and small, $0=z(z-y^2+t)=0$  with two singular points with multiplicity and  Milnor number $m=2, \mu=1$ each. 
 Thus, the sum of the multiplicities  at its zero set is
\begin{itemize}
  \item
for $t=0, \mult \Delta = 3+(2-1)=4$, 
\item
for $t\ne 0, \mult\Delta ^1+ \mult\Delta^2=1+(2-1)+1+(2-1)=4$.
\end{itemize}
Thus, the constancy (with respect to the parameter $t$)  of the sum of the multiplicities of the discriminant does not guarantee the non  splitting of the singular locus, neither the constancy of the sum of the Milnor numbers. 
\newpage

\noindent
{\bf Some other properties:}
\begin{enumerate}
\item
The limits of $\left [\frac{\partial F}{\partial x}:\frac{\partial F}{\partial y}:\frac{\partial F}{\partial z}:\frac{\partial F}{\partial t}\right]_{|\frakX} \to [\eta_1:\eta_2:\eta_3:\eta_4]\in \check {\mathbb P}^3$ at the origin are of the following form.  The last component $\eta_4$ is $0$ and $[\eta_1:\eta_2:\eta_3]$ belongs to the dual of $C_{X_1,0}$, i.e. $[0:0:1]$, or to the dual of $C_{X_2,0}$, i.e. it satisfies $4\eta_1\eta_3-\eta_2^2=0$. 
\item
The relative polar curve 
$ \overline{\left \{\frac {\partial F} {\partial x}=\frac {\partial F} {\partial y}=\frac {\partial F} {\partial z} =0, F\ne 0 \right\}}$ is empty.
\item
The inequality 
$|\frac {\partial F} {\partial t}| \leq C |\frac {\partial F} {\partial x} ,\frac {\partial F} {\partial y} ,\frac {\partial F} {\partial z} |$
fails.
\item 
Every plane of $\check{\mathbb P}^3$ is the limit of the (projectivised) gradient 
$$
 \left [\frac {\partial F} {\partial x} :  \frac {\partial F} {\partial y} :  \frac {\partial F} {\partial z} :  \frac {\partial F} {\partial t} \right ] .$$
That is, $0\times \check{\mathbb P}^3$ is a component of the exceptional divisor of the blowing-up of the jacobian ideal 
$\bigl ( \frac { \partial F} {\partial x} ,\frac {\partial F} {\partial y} , \frac {\partial F} {\partial z}, \frac {\partial F} {\partial t} \bigr )$. 
\end{enumerate} 

\begin {proof}
\begin{enumerate}
\item
For $g=zx-y^2, \eta_1=\frac{\partial g}{\partial x}=z, \eta_2=\frac{\partial g}{\partial y}=-2y, \eta_3=\frac{\partial g}{\partial z}=x$ the limits satisfy the equation of the dual curve $4\eta_1\eta_3-\eta_2^2=4zx-4y^2=0$.
\item  
We show that $\{\frac { \partial F} {\partial x} =\frac {\partial F} {\partial y} = \frac {\partial F} {\partial z}=0\}\subset F=0$.  
If $\frac {\partial F} {\partial y}=-2yz=0$, then either $z=0$ or $y=0$. 
If $z=0$ then $F=0$. 

Suppose $y=0$.  Then $x\frac {\partial F} {\partial x}=z^2x+3tzx^3=0, z\frac {\partial F} {\partial z}=2z^2x+tzx^3=0$.  
This gives $z^2x=0$. If $z=0$ then $F=0$. If $x=0$, then $\frac {\partial F} {\partial x}=z^2=0$ and hence again $z=0$.
\item 
It fails on the curve $y=0, 2z+tx^2=0, z=x^4, t=-2x^2.$
Indeed, 
\begin{align*}
 \left |\frac {\partial F} {\partial x}, \frac {\partial F} {\partial y}, \frac {\partial F} {\partial z} \right | &=|z^2+3tzx^2,0,0|=|x^8-2x^8| \\ 
& = o\left (\left |\frac {\partial F} {\partial t}\right |=|zx^3|=|x^7| \right ). 
\end{align*} 
\item 
On the curve $y=\gamma x^3+ \cdots , z=\alpha_1x^3+\alpha_2x^4+\alpha_3x^5+ \cdots , t=\beta_1x+\beta_2 x^2+\beta_3x^3+ \cdots $, parameterized by $x$,  if we assume $2\alpha_1+\beta_1=2\alpha_2+\beta_2=0$, we have 
\begin{align*}
 \frac {\partial F} {\partial x} & =z^2+3tzx^2=(\alpha_1^2+3\alpha_1\beta_1)x^6+ \cdots =-5\alpha_1^2x^6+ \cdots , \\
 \frac {\partial F} {\partial y} & =-2yz=-2\alpha_1\gamma x^6+ \cdots,\\
   \frac {\partial F} {\partial z} & =2zx -y^2+tx^3  \\ 
& =(2\alpha_1+\beta_1)x^4+ (2\alpha_2  +\beta_2)x^5  +(2\alpha_3+\beta_3-\gamma^2)x^6+ \cdots \\ 
 & =(2\alpha_3+\beta_3-\gamma^2)x^6+ \cdots, \\
  \frac {\partial F} {\partial t} & =zx^3=\alpha_1x^6+ \cdots . 
\end{align*}
Then, $ \left [\frac {\partial F} {\partial x} :  \frac {\partial F} {\partial y} :  \frac {\partial F} {\partial z} :  \frac {\partial F} {\partial t} \right ] \to \left[ -5\alpha_1: -2\gamma : \alpha_1^{-1}(2\alpha_3+\beta_3-\gamma^2) : 1\right]$. \\
Choosing $\alpha_1, \gamma, \alpha_3, \beta_3$ we can get as limits all  points of a dense set of $\mathbb P^3$, and of course, the set of limits is closed.
\end{enumerate}
\end{proof}


\subsection*{Acknowledgements}
The authors would like to thank Bernard Teissier for explaining why and how the original example of his note \cite{Te90} has to be corrected, and Jean-Baptiste Campesato and Jawad Snoussi for help in the preparation of this paper.

\subsection*{Funding}
{The first author is grateful for the support and hospitality of the Sydney Mathematical Research Institute (SMRI). 
The second author  acknowledges support from the Project ‘Singularities and Applications’-CF
132/31.07.2023 funded by the European Union-NextGenerationEU-through Romania’s National
Recovery and Resilience Plan.
}


\end{document}